\documentclass[10pt]{amsart}
\usepackage{amsthm, amsfonts, amssymb, amsmath, amscd}
\usepackage[all]{xy}
\usepackage[margin=3.5cm]{geometry}

\newcommand{\C}{\mathbb{C}}
\newcommand{\R}{\mathbb{R}}

\newcommand{\G}{\Gamma}
\newcommand{\F}{\mathbb{F}_n}

\newcommand{\e}{\varepsilon}
\newcommand{\noin}{\noindent}
\newcommand{\la}{\langle}
\newcommand{\ra}{\rangle}
\newcommand{\bd}{\partial}

\newcommand{\dhat}{\hat{d}}
\newcommand{\lhat}{\hat{l}}
\newcommand{\D}{\:\mathrm{d}}
\newcommand{\diam}{\mathrm{diam}\:}
\newcommand{\dist}{\mathrm{dist}}
\newcommand{\Mob}{\mathrm{M\ddot{o}b}}
\newcommand{\cleq}{\preccurlyeq}
\newcommand{\cgeq}{\succcurlyeq}

\newtheorem{thm}{Theorem}[section]
\newtheorem{lem}[thm]{Lemma}
\newtheorem{cor}[thm]{Corollary}
\newtheorem{prop}[thm]{Proposition}

\theoremstyle{definition}

\newtheorem{rem}[thm]{Remark}
\newtheorem{ex}[thm]{Example}

\begin{document}
\title{Proper isometric actions of hyperbolic groups on $L^p$-spaces}
\author{Bogdan Nica}
\address{Mathematisches Institut, Georg-August Universit\"at G\"ottingen\newline \phantom{MM} Bunsenstrasse 3--5, D-37073 G\"ottingen, Germany}
\email{bogdan.nica@gmail.com}
\date{\today \\
\phantom{nn} 2010 Mathematics Subject Classification: 20F67 (primary), 20J06, 30L10, 43A15 (secondary)\\
\phantom{nn} Keywords: isometric actions on Banach spaces, hyperbolic groups, Bowen - Margulis measure, $\ell^p$-cohomology}
\begin{abstract}
\noin We show that every non-elementary hyperbolic group $\G$ admits a proper affine isometric action on $L^p(\bd\G\times \bd\G)$, where $\bd\G$ denotes the boundary of $\G$ and $p$ is large enough. Our construction involves a $\G$-invariant measure on $\bd\G\times \bd\G$ analogous to the Bowen--Margulis measure from the CAT$(-1)$ setting, as well as a geometric, Busemann-type cocycle. We also deduce that $\G$ admits a proper affine isometric action on the first $\ell^p$-cohomology group $H^1_{(p)}(\G)$ for large enough $p$.
\end{abstract}
\maketitle

\section{Introduction}
With respect to the geometry of $L^2$-spaces, the class of hyperbolic groups appears to be both ``soft'' and ``rigid''. This ambivalence is vividly illustrated by cocompact lattices in isometry groups of rank-1 symmetric spaces. For real or complex hyperbolic spaces, cocompact lattices admit proper isometric actions on Hilbert spaces. For quaternionic hyperbolic spaces or the octonionic hyperbolic plane, every isometric action of a cocompact lattice on a Hilbert space is bounded. In other words, cocompact lattices enjoy the Haagerup property, or a-T-menability, in the real or complex case (see \cite{CCJJV}), respectively Kazhdan's property (T), or Serre's property FH, in the quaternionic or octonionic case (see \cite{BHV}).

When tested against general $L^p$-spaces, hyperbolic groups reveal themselves to be ``soft'': every hyperbolic group admits a proper isometric action on an $L^p$-space, where $p$ depends on the group. This fact, due to Yu, is one of the most interesting results in the study of isometric group actions on uniformly convex Banach spaces. More precisely, the following is shown in \cite{Yu}: 

\begin{thm}[Yu]\label{yu's theorem}
Let $\G$ be a hyperbolic group. Then, for large enough $p$, the linear isometric action of $\G$ on $\ell^p(\G\times\G)$ admits a proper cocycle.
\end{thm}

If we let a non-elementary hyperbolic group $\G$ act on its boundary $\bd\G$ rather than on itself, then a number of finiteness properties emerge at infinity. An example is the fact that, although $\G$ is non-amenable, the action of $\G$ on $\bd\G$ is amenable (S. Adams \cite{Ada}). In a recent joint work with Emerson \cite{EN}, we construct Fredholm modules for the $\mathrm{C}^*$-crossed product $C(\bd \G)\rtimes \G$ which are $p$-summable for every $p\in (2,\infty)$ greater than the visual dimension of $\bd\G$. Very informally, this means that the action of a hyperbolic group on its boundary is summable above the Hausdorff dimension of the boundary. This finiteness phenomenon is the inspiration for the following boundary analogue of Yu's theorem:

\begin{thm}\label{main} 
Let $\G$ be a non-elementary hyperbolic group and let $\bd\G$ denote its boundary. Then, for large enough $p$, the linear isometric action of $\G$ on $L^p(\bd\G\times\bd\G)$ admits a proper cocycle.
\end{thm}

On the surface, Theorem~\ref{main} conveys the same idea as Theorem~\ref{yu's theorem}, namely that hyperbolic groups admit proper isometric actions on $L^p$-spaces. The two theorems are nevertheless different both in statement and in proof. Our approach has the following novel features.

\smallskip \emph{Framework.} In the first part of the paper, we put forth the following principle: a M\"obius action on an Ahlfors regular, compact metric space gives rise to an affine isometric action on an $L^p$-space for each $p$ greater than the Hausdorff dimension of the metric space. See Section \ref{M-actions}. 

\smallskip \emph{Measure.} For the action of a non-elementary hyperbolic group $\G$ on its boundary $\bd\G$, the M\"obius philosophy yields two ingredients. The first is an explicit $\G$-invariant measure on $\bd\G\times \bd\G$, twin to the Bowen--Margulis measure encountered in the CAT$(-1)$ setting. Our generalized Bowen--Margulis measure significantly improves a previous construction by Furman \cite[Prop.1]{Fu}. See Section \ref{the paragraph on measure}.

\smallskip \emph{Cocycle.} The second ingredient is a beautiful geometric cocycle for the action of $\G$ on $\bd\G\times \bd\G$. Its memorable form suggests that it could be interpreted as the \emph{other} Busemann cocycle. The properness of this cocycle depends on the hyperbolicity of $\G$. See Section \ref{the paragraph on cocycle}.

\smallskip \emph{Exponent.} The final conceptual advantage of the M\"obius philosophy is that it provides an integrability exponent $p$ which is related to a suitable interpretation of the Hausdorff dimension of the boundary $\bd\G$, the \emph{hyperbolic dimension} introduced by Mineyev \cite{M}. Alternatively, and somewhat less sharply, $p$ is related to a modified growth exponent of $\G$. See Section \ref{the paragraph on dimension}.

\smallskip \emph{$\ell^p$-cohomology.} Our construction of a proper isometric action of $\G$ on $L^p(\bd\G\times\bd\G)$ has the following $\ell^p$-cohomological interpretation:

\begin{thm}\label{lp consequence}
Let $\G$ be a non-elementary hyperbolic group. Then, for large enough $p$, the linear isometric action of $\G$ on the first $\ell^p$-cohomology group $H^1_{(p)}(\G)$ admits a proper cocycle.
\end{thm}

Again, we provide an appealing cocycle for the action, see Section~\ref{lp-cohomology stuff}. Typical results on the first $\ell^p$-cohomology of finitely generated groups are concerned with the vanishing/non-vanishing dichotomy, and Theorem~\ref{lp consequence} is entirely new in that respect. Its proof uses a result of Bourdon and Pajot \cite{BP}.

Differences aside, the proofs of Theorem~\ref{main} and Theorem~\ref{yu's theorem} share a common technical point, and that is Mineyev's powerful re-metrization procedure. For our purposes, the relevant upshot of this procedure is that it leads to visual metrics on the boundary with much better properties than the visual metrics coming from the word metric. These new visual metrics were constructed by Mineyev in \cite{M}.

We would like to mention another recent result concerning proper isometric actions of hyperbolic groups on $L^p$-spaces. In \cite{B+}, Bourdon shows the following: for every non-elementary hyperbolic group $\G$, there is a positive integer $n$ such that the linear isometric action of $\G$ on $\ell^p(\sqcup_1^n\G)$ admits a proper cocycle for every $p>\mathrm{Confdim}\; \bd \G$. Here, too, the integrability exponent $p$ is related to a suitable interpretation of the Hausdorff dimension of the boundary: the Ahlfors regular, conformal dimension $\mathrm{Confdim}$, which is easily seen to be no larger than the hyperbolic dimension mentioned above. Thus, Bourdon's exponent bound is in principle better than ours, though discriminating examples are probably very hard to construct (if there are any at all). Note however that the linear part of the action requires several copies of the regular representation, and there seems to be no explicit formula for the number of copies one needs to consider. Bourdon's construction of proper isometric actions is achieved via $\ell^p$-cohomology; in particular, it has the advantage of dispensing with Mineyev's technical procedure.

\section{Preliminaries}
\subsection{Notation} We write $\cleq$ to mean inequality up to a positive multiplicative constant, and the corresponding equivalence is denoted $\asymp$. The constants involved in these relations often depend on some parameter, and we record this dependence as a subscript (e.g., $\cleq _{\e}$).

\subsection{Isometric actions on $L^p$-spaces}\label{general} 
By the Mazur--Ulam theorem (see \cite{Nic12} for a short proof) isometric group actions on real Banach spaces are affine. An isometric action of a discrete group $\G$ on a real $L^p$-space is obtained from two ingredients. The first ingredient is a measured space $X$ on which $\G$ acts in a measure-preserving way; then $(g,F)\mapsto g. F$ is a linear isometric action of $\G$ on $L^p(X)$. The second ingredient is a cocycle for this linear isometric action, that is a map $c:\G\to L^p(X)$ satisfying $c_{gh}=g.c_h+c_g$ for all $g,h\in \G$; then $(g,F)\mapsto g. F+c_g$ is an affine isometric action of $\G$ on $L^p(X)$. The latter isometric action is proper if $\|c_g\|_{L^p(X)}\to \infty$ as $g\to \infty$ in $\G$.

Throughout the paper, it is understood that $p\in [1,\infty)$.

\section{Proper isometric actions of free groups on $L^p$-spaces}\label{free group case}
As a warm-up, we start by giving a simple and self-contained proof of Theorem~\ref{main} in the case of free groups. The purpose of this discussion is to foreshadow two key points, developed later: the construction of affine isometric actions on $L^p$-spaces from M\"obius actions (Section~\ref{M-actions}), and the properness of the cocycle in the case of hyperbolic groups (Section~\ref{hyperbolic case}).

\subsection{The boundary $\Omega$} Let $\F$ be the free group on $n\geq 2$ generators, and put $q=2n-1$. The Cayley graph of $\F$ with respect to the standard generators is the $2n$-valent tree, rooted at the identity element. The boundary of this tree, customarily denoted by $\Omega$, is the set of all infinite rooted paths without backtracks. We endow $\Omega$ with a probability measure $\mu$ defined by the requirement that
 \[\mu(\Omega_x)=\tfrac{q}{q+1}\: q^{-|x|}\]
for all non-identity elements $x\in \F$. Here $\Omega_x$ denotes the ``boundary under $x$'', that is, the boundary subset consisting of all those $\omega\in \Omega$ which start with $x$, and $|x|$ is the length of $x$.

The Poisson kernel is given by
\[P_g(\omega)=q^{-|g|+2(g,\omega)}\qquad (g\in\F, \omega\in \Omega)\]
where $(\cdot, \cdot)$, the Gromov product based at the identity, measures the length of the longest shared path. The Poisson kernel $P_g$ represents the Radon - Nikodym derivative $\mathrm{d}(g_*\mu)/\mathrm{d}\mu$, and it satisfies the cocycle relation $P_{g h}=P_g\; g. P_{h}$ for $g,h\in \F$. (For more details see, for instance, Fig\`a-Talamanca - Picardello \cite{FTP}.)

\subsection{An invariant measure on $\Omega\times \Omega$} The following key relation is easy to check: 
\begin{align*}\tag{$\dagger$}q^{-2(g \xi, g\omega)}= P_{g^{-1}}(\xi)\:P_{g^{-1}}(\omega)\:q^{-2(\xi, \omega)}
\end{align*}
for all $g\in \F$ and $\xi, \omega\in \Omega$. The main consequence of this relation is that a suitably weighted product measure on $\Omega\times \Omega$ is $\F$-invariant. Namely, let $\nu$ be the measure on $\Omega\times \Omega$ given by
\[\D\nu = q^{2(\xi,\omega)}\D\mu(\xi)\D\mu(\omega).\]

The diagonal of $\Omega\times \Omega$ is $(\mu\times \mu)$-negligible, since points of $\Omega$ are $\mu$-negligible. Therefore $\nu$ is well-defined, and the diagonal of $\Omega\times \Omega$ is $\nu$-negligible as well. 

Note also that $\nu$ is infinite and $\sigma$-finite. Indeed, consider the countable partition of $\Omega \times \Omega- \mathrm{diag}$ given by the sets $K_n=\{(\xi,\omega )\in  \Omega\times \Omega: (\xi,\omega)=n\}$, where $n\geq 0$. We claim that each $K_n$ has finite $\nu$-measure, and $\sum \nu(K_n)=\infty$. When $n\geq 1$, we have $\mu(\{\xi\in  \Omega: (\xi,\omega)=n\})=\tfrac{q-1}{q+1}\: q^{-n}$ for each $\omega\in\Omega$. Then $(\mu\times\mu)(K_n)=\tfrac{q-1}{q+1}\: q^{-n}$, hence $\nu(K_n)=\tfrac{q-1}{q+1}\: q^{n}$. A similar argument shows that $\nu(K_0)=\tfrac{q}{q+1}$. 

\begin{lem}\label{invariance for free} The measure $\nu$ is invariant for the diagonal action of $\F$.
\end{lem}
\begin{proof}
 Let $F\in L^1(\Omega\times \Omega, \nu)$ and $g\in \F$. Then:
\begin{align*}
\int g . F \D \nu &=\iint F(g^{-1}\xi,g^{-1} \omega)\: q^{2(\xi,\omega)}\D\mu(\xi)\D\mu(\omega)\\
&= \iint F(\xi, \omega)\: q^{2(g\xi,g\omega)}\D g^*\mu(\xi)\D g^*\mu(\omega)\\
&= \iint F(\xi, \omega)\: q^{2(g\xi,g\omega)} P_{g^{-1}}(\xi)P_{g^{-1}}(\omega) \D\mu(\xi)\D\mu(\omega)\\
&=\iint F(\xi, \omega)\: q^{2(\xi,\omega)} \D\mu(\xi)\D\mu(\omega)= \int F \D \nu
\end{align*}
The above computation involves a change of variables, followed by an application of $(\dagger)$.
\end{proof}

\subsection{A cocycle for the action on $\Omega\times \Omega$} A natural cocycle for the action of $\F$ on $\Omega$ is the logarithm of the Poisson kernel $g\mapsto \log P_{g}$, and a cocycle in two variables can be obtained by taking the difference of two such cocycles. More precisely, if
\[c_g(\xi,\omega):=\tfrac{1}{2}\big(\log P_{g}(\xi)- \log P_{g}(\omega) \big)=(g,\xi)-(g,\omega)\]
then $g\mapsto c_g$ a cocycle for the diagonal action of $\F$ on $\Omega\times \Omega$. The next proposition shows that, for each $p$, $c$ is a proper cocycle for the linear isometric representation of $\F$ on $L^p(\Omega\times \Omega, \nu)$.

\begin{prop}\label{properness for free} We have $\|c_g\|_{L^p(\nu)}\asymp_p |g|^{1/p}$.
\end{prop}

\begin{proof} Let $g\in \F$ and write
\[\|c_g\|^p_{L^p(\nu)}=\iint \big |(g,\xi)-(g,\omega)\big|^p\: q^{2(\xi,\omega)}\D\mu(\xi)\D\mu(\omega).\]
As $\xi$ runs over $\Omega$, the Gromov product $(g, \xi)$ takes on the values $0,1,\dots, |g|$. Thus
\begin{align*}
\|c_g\|^p_{L^p(\nu)}&=\sum _{i,j=0}^{|g|} \iint_{\begin{smallmatrix} (g,\xi)=i \\ (g,\omega)=j \end{smallmatrix}} \big |(g,\xi)-(g,\omega)\big|^p\: q^{2(\xi,\omega)}\D\mu(\xi)\D\mu(\omega)\\
&=\sum _{i,j=0}^{|g|}  |i-j|^p\: q^{2\min\{i,j\}}  \mu\big( \{\xi: (g,\xi)=i\}\big)\: \mu\big(\{\omega:  (g,\omega)=j\}\big),
\end{align*}
using the fact that $(\xi,\omega)=\min\{(g,\xi),(g,\omega)\}$ whenever $(g,\xi)\neq (g,\omega)$. For $i\in \{0,1,\dots, |g|\}$, we have $\mu\big(\{\xi: (g, \xi)=i\}\big )\asymp q^{-i}$. In fact,
\begin{align*}
\tfrac{q-1}{q+1}q^{-i}\leq \mu\big(\{\xi: (g, \xi)=i\}\big)\leq \tfrac{q}{q+1}q^{-i}
\end{align*}
with equality on the left for all $i\neq 0,|g|$, and equality on the right at the endpoints $0,|g|$. Hence, letting
\begin{align*}
S_N:&=\sum _{i,j=0}^{N}  |i-j|^p\: q^{2\min\{i,j\}} q^{-i} q^{-j}= \sum _{i,j=0}^{N}  |i-j|^p\: q^{-|i-j|}
\end{align*}
we have $\|c_g\|^p_{L^p(\nu)}\asymp S_{|g|}$. Now $S_N\asymp_p N$: the recurrence
\begin{align*}
S_{N+1}& =\sum _{i,j=0}^{N+1}  |i-j|^p\: q^{-|i-j|}=S_N+2 \sum _{i=0}^{N}  (N+1-i)^p\: q^{-(N+1-i)}=S_N+2 \sum _{i=1}^{N+1}  i^p\: q^{-i}
\end{align*}
gives $S_N+2q^{-1}\leq S_{N+1}<S_N+2\big(\sum _{i\geq 1}  i^p\: q^{-i}\big)$. We conclude that $\|c_g\|^p_{L^p(\nu)}\asymp_p |g|$.
\end{proof}

There are, certainly, easier ways to produce proper isometric actions of free groups on $L^p$-spaces. Here is one such action. Let $\vec{X}$ be the directed Cayley graph of $\F$ with respect to the standard generators. Let $\{g\to h\}$ denote the shortest oriented path in $\vec{X}$ from $g\in \F$ to $h\in \F$. Then $g\mapsto c_g:=\big(\sum_{e\in \{1\to g\}} e-\sum_{e\in \{g\to 1\}} e\big)$
defines a proper cocycle for the linear isometric action of $\F$ on the $\ell^p$-space of the edge-set of $\vec{X}$.

Our aim in what follows is to promote the boundary-based method presented in this section to general non-elementary hyperbolic groups.

\section{Interlude: M\"obius calculus}\label{M-calculus}
In this section we discuss ``derivatives'' of M\"obius maps. The facts established herein will be used in the next section to construct affine isometric actions of M\"obius groups on $L^p$-spaces.

Throughout, we let $(X,d)$ be a compact metric space without isolated points and we consider M\"obius self-homeomorphisms of $X$. Recall that the \emph{cross-ratio} of a quadruple of distinct points in $X$ is defined by the formula
\[(z_1,z_2;z_3,z_4)=\frac{d(z_1,z_3)\: d(z_2,z_4)}{d(z_1,z_4)\: d(z_2,z_3)}.\]
A homeomorphism $g: X\to X$ is called a \emph{M\"obius homeomorphism} if $g$ preserves the cross-ratios, i.e., $(g z_1,g z_2;gz_3,g z_4)=(z_1,z_2;z_3,z_4)$ for all quadruples of distinct points $z_1,z_2,z_3,z_4\in X$.

\begin{lem}\label{alternate} Let $g$ be a self-homeomorphism of $X$. Then $g$ is M\"obius if and only if there exists a positive continuous function on $X$, denoted $|g'|$, with the property that for all $x, y\in X$ we have
\begin{align*}\tag{$*$}
d^{\:2}(g x,g y)=|g'|(x)\:|g'|(y)\:d^{\:2}(x,y).
\end{align*}
\end{lem}

Before we prove the lemma, let us observe that a continuous function $|g'|$ satisfying $(*)$ has, in particular, the property that
\begin{align*}
\lim_{y\to x} \frac{d(gx,gy)}{d(x,y)}=|g'|(x)
\end{align*}
for all $x\in X$. This property justifies the notation, as well as the interpretation of $|g'|$ as the \emph{metric derivative} of $g$. Following Sullivan \cite[Sec.4]{S}, we also interpret the relation $(*)$ as a geometric mean-value property.

\begin{proof}[Proof of Lemma~\ref{alternate}] 
We prove the forward implication. The converse is a trivial verification.

Assume that $g$ is a M\"obius homeomorphism. Let $x,u,v$ be a triple of distinct points in $X$. For any fourth distinct point $y$ we have
\[\frac{d(gx,gy)}{d(x,y)}\: \frac{d(gu,gv)}{d(u,v)}=\frac{d(gx,gu)}{d(x,u)}\: \frac{d(gy,gv)}{d(y,v)}\]
since $g$ preserves the cross-ratios. When $y\to x$, we obtain
\begin{align}\label{complicated derivative}
\lim_{y\to x} \frac{d(gx,gy)}{d(x,y)}=\frac{d(gx,gu)}{d(x,u)}\: \frac{d(gx,gv)}{d(x,v)}\: \frac{d(u,v)}{d(gu,gv)}.
\end{align}
Let $|g'|_{u,v}$ denote the expression on the right-hand side of \eqref{complicated derivative}, viewed as a function of $x$. Then $|g'|_{u,v}$ is a positive continuous function on $X- \{u,v\}$. However, the left-hand side of \eqref{complicated derivative} is independent of the choice of $u,v$. Thus, picking $\tilde u,\tilde v$ distinct points in $X- \{u,v\}$, we have that $|g'|_{u,v}=|g'|_{\tilde u,\tilde v}$ on $X- \{u,\tilde u,v,\tilde v\}$. Defining $|g'|$ on $X$ as $|g'|_{u,v}$ on $X- \{u,v\}$, and $|g'|_{\tilde u,\tilde v}$ on $X- \{\tilde u,\tilde v\}$, we obtain a positive continuous function. 

Now let us prove $(*)$ for distinct $x,y\in X$. Let $u,v$ be distinct points in $X- \{x,y\}$, so that we can use the local formula $|g'|_{u,v}$ for $|g'|$. The equality $d^{\:2}(gx,gy)=|g'|_{u,v}(x)\:|g'|_{u,v}(y)\: d^{\:2}(x,y)$ can be readily checked by rearranging factors and using the $g$-invariance of the cross-ratios. 
\end{proof}

The next lemma shows that metric derivatives are more than just continuous. 

\begin{lem}\label{Lipschitz}
Let $g$ be a M\"obius self-homeomorphism of $X$. Then $|g'|$ is Lipschitz.
\end{lem}

\begin{proof} We show that $\sqrt{|g'|}$ is Lipschitz. This is equivalent to $|g'|$ being Lipschitz, since 
\[2\: \sqrt{\min |g'|}\:\Big|\sqrt{|g'|}(x)-\sqrt{|g'|}(y)\Big|\leq \big||g'|(x)-|g'|(y)\big|\leq 2\: \sqrt{\max |g'|}\:\Big|\sqrt{|g'|}(x)-\sqrt{|g'|}(y)\Big|.\]

Let $x,y \in X$. There exists $z\in X-\{y\}$ such that $d(x,z)\geq (\diam X)/2$. By the geometric mean-value property $(*)$, we have 
\[\sqrt{|g'|}(x)-\sqrt{|g'|}(y)=\frac{1}{\sqrt{|g'|}(z)}\bigg(\frac{d(gx,gz)}{d(x,z)}-\frac{d(gy,gz)}{d(y,z)}\bigg).\]
Using the fact that $g$ is $(\max |g'|)$-Lipschitz, we estimate 
\begin{align*}
\frac{d(gx,gz)}{d(x,z)}-\frac{d(gy,gz)}{d(y,z)}& \leq \frac{d(gx,gy)}{d(x,z)}+\frac{d(gy,gz)}{d(x,z)}-\frac{d(gy,gz)}{d(y,z)}\\
& \leq \frac{d(gx,gy)}{d(x,z)}+\frac{d(gy,gz)}{d(y,z)}\:\frac{d(x,y)}{d(x,z)}\\
& \leq 2\: (\max |g'|)\:\frac{d(x,y)}{d(x,z)} \leq \frac{4\: (\max |g'|)}{\diam X}\: d(x,y).
\end{align*}
Thus
\begin{align}\label{lipschitz estimate}
\sqrt{|g'|}(x)-\sqrt{|g'|}(y)\leq \frac{4}{\diam X}\:\frac{\max |g'|}{\sqrt{ \min |g'|}}\:d(x,y)
\end{align}
so $\sqrt{|g'|}$ is, indeed, Lipschitz.
\end{proof}

We now take a measure-theoretic turn. Recall that the Hausdorff measure of dimension $D\geq 0$ on $X$ is defined by the formula
\[\mu_D(A)=\lim _{\delta\to 0} \Big(\inf\Big\{\sum (\diam U_i)^D\: : \: (U_i)\;  \delta\textrm{-cover of } A\Big\}\Big)\qquad (A\subseteq X).\]
This is a Borel measure which is interesting for a single $D$ only, the Hausdorff dimension of $X$.

The following lemma is a variation on a basic observation of Sullivan (cf. \cite[p.174]{S}).

\begin{lem}\label{Radon Nikodym}
Assume that $X$ has finite, non-zero Hausdorff dimension $D$, and let $g$ be a M\"obius self-homeomorphism of $X$. Then $|g'|^D$ represents the Radon - Nikodym derivative $\mathrm{d} g^*\mu_D/\mathrm{d} \mu_D$.
\end{lem}

\begin{proof} Fix a measurable subset $S\subseteq X$. We want to show that 
\begin{align}\label{desire}
\mu_D(gS)=\int_S |g'|^D \D\mu_D.
\end{align}
From the geometric mean-value property $(*)$, we get that 
\[\big (\inf_U |g'|\big)\:\diam U\leq \diam gU\leq \big (\sup_U |g'|\big)\:\diam U\] 
for every $U\subseteq X$. It follows that, for every measurable $T\subseteq X$, we have
\begin{align}\label{Lipschitz type}
\big(\inf_T |g'|\big)^D\: \mu_D(T)\leq \mu_D(gT)\leq \big(\sup_T |g'|\big)^D\:\mu_D(T).
\end{align}
Now let $\e>0$. Let also $\eta>0$ be such that $\big||g'|(x)-|g'|(y)\big|<\e \min |g'|$ whenever $d(x,y)<\eta$. Hence, if $T$ is a measurable $\eta$-set, in the sense that $\diam T<\eta$, then $\sup_T |g'|\leq (1+\e) \inf_T |g'| $, which in turn yields
\[(1+\e)^{-D}\:\big(\sup_T |g'|\big)^D\:\mu_D(T)\leq \int_T |g'|^D \D\mu_D\leq (1+\e)^D\:\big(\inf_T |g'|\big)^D\:\mu_D(T).\]
Thus, in light of \eqref{Lipschitz type}, we get that for every measurable $\eta$-set $T$ the following holds:
\begin{align}\label{small diameter}
(1+\e)^{-D}\: \mu_D(gT)\leq \int_T |g'|^D \D\mu_D\leq (1+\e)^D\: \mu_D(gT)
\end{align}
The measurable set $S$ we started with may be partitioned into a finite number of measurable $\eta$-subsets. (Indeed, pick a finite cover of $X$ by open $\eta$-subsets. This cover gives rise to a finite partition of $X$ into Borel $\eta$-subsets, which partition can be used on any measurable subset of $X$.) Applying \eqref{small diameter} to these $\eta$-pieces of $S$, and then adding up, we get
\begin{align*}
(1+\e)^{-D}\: \mu_D(gS)\leq \int_S |g'|^D \D\mu_D\leq (1+\e)^D\: \mu_D(gS).
\end{align*}
Since $\e$ is arbitrary, we conclude that \eqref{desire} holds.
\end{proof}

\section{Affine isometric actions of M\"obius groups on $L^p$-spaces}\label{M-actions}
As before, $(X,d)$ is a compact metric space without isolated points. Throughout this section we also assume that the Hausdorff dimension of $X$, denoted by $D$, is finite and non-zero. The goal is to construct an affine isometric action of $\Mob(X)$, the group of M\"obius self-homeomorphisms of $X$, on $L^p(X\times X)$. 

\subsection{A M\"obius-invariant measure on $X\times X$} The relevant measure on $X\times X$ is given by a suitable weighting of the product measure $\mu_D\times \mu_D$, where $\mu_D$ is the $D$-dimensional Hausdorff measure on $X$.

\begin{lem}\label{nature of nu} Assume $\mu_D(X)<\infty$. Then
\[\D \nu(x,y)=d^{\: -2D}(x,y) \D\mu_D (x)\D\mu_D (y)\]
defines a $\sigma$-finite Borel measure on $X\times X$. The diagonal of $X\times X$ is $\nu$-negligible, and on the locally compact and $\sigma$-compact space $X\times X- \mathrm{diag}$, the measure $\nu$ is a Radon measure.
\end{lem}
\begin{proof}
By the separability of $X$, the product measure $\mu_D\times \mu_D$ is a Borel measure on $X\times X$. Points are negligible for a Hausdorff measure of positive dimension, so the diagonal of $X\times X$ is ($\mu_D\times \mu_D$)-negligible. Thus $\nu$ is well-defined, and the diagonal of $X\times X$ is $\nu$-negligible. Since $\nu$ is obtained by weighting a Borel measure, namely $\mu_D\times \mu_D$, by a Borel map, namely $d^{\: -2D}(\cdot,\cdot)$, it follows that $\nu$ is Borel on $X\times X$. The $\sigma$-finiteness of $\nu$, as well as the $\sigma$-compactness of the ``slashed square'' $X\times X- \mathrm{diag}$, follow by writing $X\times X- \mathrm{diag}=\bigcup_{n\geq 1}  \big\{(x,y): d(x,y)\geq 1/n\big\}$. When restricted to $X\times X- \mathrm{diag}$, $\nu$ is a Borel measure which is finite on compact subsets. The regularity of $\nu$ is automatic: it follows from \cite[Thm.2.18]{Ru} that, on a locally compact and $\sigma$-compact metric space, a Borel measure which is finite on compact sets is a Radon measure.
\end{proof}

The key property of $\nu$ is its M\"obius invariance.

\begin{lem}\label{split conformal}
Assume $\mu_D(X)<\infty$. Then $\nu$ is invariant for the diagonal action of $\Mob(X)$.
\end{lem}

\begin{proof}
By the $\sigma$-finiteness of $\nu$, it suffices to show that sets of finite measure are invariant. This is shown as in the proof of Lemma~\ref{invariance for free}. For $F \in L^1(X\times X, \nu)$ and $g\in\Mob(X)$ we have:
\begin{align*}
\int g . F \D \nu &=\iint F(g^{-1}x,g^{-1} y)\: d^{\: -2D}(x,y)\D\mu_D(x)\D\mu_D(y)\\
&= \iint F(x, y)\: d^{\: -2D}(gx,gy)\D g^*\mu_D(x)\D g^*\mu_D(y)\\
&= \iint F(x, y)\: d^{\: -2D}(gx,gy)\: |g'|^D(x)\:|g'|^D(y) \D\mu_D(x)\D\mu_D(y)\\
&=\iint F(x, y)\: d^{\: -2D}(x,y) \D\mu_D(x)\D\mu_D(y)= \int F \D \nu
\end{align*}
The second equality is a change of variables, the third relies on Lemma~\ref{Radon Nikodym}, and the fourth is due to the geometric mean-value property $(*)$.
\end{proof}

\subsection{A cocycle for the M\"obius action on $X\times X$} At this point, we have a linear isometric action of $\Mob(X)$ on $L^p(X\times X,\nu)$ for each $p$. We need a cocycle in order to get an affine isometric action, and this arises as follows.  The metric derivatives satisfy the chain rule 
\begin{align}\label{chain rule}
|(gh)'|(x)=|g'|(hx)\: |h'|(x)
\end{align}
for all $g,h\in \Mob(X)$ and $x\in X$. In other words, $g\mapsto\log \big|(g^{-1})'\big|$ is a cocycle for the action of $\Mob(X)$ on $X$. Hence 
\[g\mapsto c_g(x,y):=\log \big|(g^{-1})'\big|(x)-\log \big|(g^{-1})'\big|(y)\]
defines a cocycle for the diagonal action of $\Mob(X)$ on $X\times X$. Now $g\mapsto c_g$ is a cocycle for the linear action of $\Mob(X)$ on $L^p(X\times X,\nu)$ if and only if each $c_g$ is in $L^p(X\times X,\nu)$. This turns out to be the case for all $p$ greater than $D$, the Hausdorff dimension of $X$, as soon as we require $\mu_D$ to be \emph{Ahlfors regular}. This means that the measure of balls, viewed as a function of the radius $r\in [0,\diam X]$, satisfies
\[\mu_D(r\mathrm{-ball}) \asymp r^D.\]  
In particular, if $\mu_D$ is Ahlfors regular then $\mu_D(X)<\infty$.

\begin{lem}\label{integrability of distance}
Assume that $\mu_D$ is Ahlfors regular. Then the metric $d$ is not in $L^D(X\times X,\nu)$, but it does belong to the weak $L^D$-space $L^{D,\infty} (X\times X,\nu)$. Consequently, $d\in L^p(X\times X,\nu)$ for each $p>D$, and $\nu$ is an infinite measure on $X\times X$.
\end{lem}

\begin{proof} To show that $d$ is in the weak $L^D$-space $L^{D,\infty} (X\times X,\nu)$, we have to check that
\[\nu\big( \{(x,y): d(x,y)>t\}\big)\cleq t^{-D}\]
as $t$ runs over positive reals. To that end, it suffices to show that for each fixed $y\in X$ we have
\[\int_{d(x,y)>t} d^{\: -2D}(x,y) \D\mu_D(x)\cleq t^{-D}\] 
independent of $y$; and indeed
\begin{align*}
\int_{d(x,y)>t} d^{\: -2D}(x,y) \D\mu_D(x)&=\int^{\: t^{-2D}}_0 \mu_D\big(\{x: d^{\: -2D}(x,y)>s\}\big) \D s \\
&=\int^{\: t^{-2D}}_{(\diam X)^{-2D}}\: \mu_D\big(\{x: d(x,y)<s^{-\frac{1}{2D}}\}\big) \D s\\
&\leq \int^{\: t^{-2D}}_{(\diam X)^{-2D}}\: Cs^{-\frac{1}{2}} \D s\leq  (2C)\: t^{-D}. 
\end{align*}
The inequality in the last line uses the upper polynomial bound on the measure of balls. Next, we show that $d$ is not in the subspace $L^D (X\times X,\nu)\subseteq L^{D,\infty} (X\times X,\nu)$, and here we use the lower polynomial bound on the measure of balls. For each $y\in X$ we have:
\begin{align*}
\int d^{\: -D}(x,y) \D\mu_D(x)&=\int_0^\infty \mu_D\big(\{x: d^{\: -D}(x,y)>s\}\big) \D s \\
&=\int^{\infty}_{(\diam X)^{-D}}\: \mu_D\big(\{x: d(x,y)<s^{-\frac{1}{D}}\}\big) \D s\\
& \geq \int^{\infty}_{(\diam X)^{-D}}\: c \:s^{-1} \D s=\infty
\end{align*}
Integrating with respect to $y$, we get $\|d\|_{L^D (\nu)}=\infty$.

The second part of the lemma follows by using the boundedness of $d$. As $d$ is in weak $L^D(\nu)$ and in $L^\infty(\nu)$, interpolation yields that $d$ is in $L^p(\nu)$ for each $p>D$. Finally, $d$ not in $L^D(\nu)$ implies in particular that $\nu$ is infinite.
\end{proof}

On the other hand, we may bound the cocycle by the metric as follows:

\begin{lem}\label{distance bounds cocycle}
We have 
\[|c_g|\cleq  \frac{\max |g'|}{\min |g'|} \: d.\]
\end{lem}

\begin{proof}
Let $g\in \Mob(X)$. Using the fact that $|\log a-\log b|\leq |a-b|/m$ whenever $a,b\geq m>0$, together with the Lipschitz estimate \eqref{lipschitz estimate}, we get
\begin{align*}
\Big|\log |g'|(x)-\log |g'|(y)\Big|&=2\:\Big|\log \sqrt{|g'|}(x)-\log \sqrt{|g'|}(y)\Big|\\
&\leq \frac{2}{\sqrt{\min |g'|}}\: \Big|\sqrt{|g'|}(x)-\sqrt{|g'|}(y)\Big|\leq \frac{8}{\diam X}\:\frac{\max |g'|}{\min |g'|} \: d(x,y).
\end{align*}
The chain rule \eqref{chain rule} implies that $|(g^{-1})'|=1/g.|g'|$, so $\max |(g^{-1})'|/\min |(g^{-1})'|=\max |g'|/\min |g'|$. Thus, the Lipschitz estimate for $\log |g'|$ becomes
\[|c_g(x,y)|\leq \frac{8}{\diam X}\:  \frac{\max |g'|}{\min |g'|} \: d(x,y),\]
as desired.
\end{proof}

Lemma~\ref{integrability of distance} and Lemma~\ref{distance bounds cocycle} imply that our cocycle $c$ takes values in $L^p(X\times X,\nu)$ for each $p>D$. Summarizing, we have proved the following:

\begin{prop}\label{cocycle} Assume that $\mu_D$ is Ahlfors regular. Then, for each $p>D$,
\[g\mapsto c_g(x,y)=\log |(g^{-1})'|(x)-\log |(g^{-1})'|(y)\]
is a cocycle for the linear isometric action of $\Mob(X)$ on $L^p(X\times X,\nu)$.
\end{prop}

\subsection{A topological perspective} The groups we are ultimately interested in, namely hyperbolic groups, are discrete. Proposition~\ref{cocycle}, in which we are treating the M\"obius group $\Mob(X)$ as a discrete group, suffices for our purposes. However, the M\"obius context we have developed so far has a topological layer as well, and we will look at it before moving on to the case of hyperbolic groups. 

Recall that the natural topology on the space $C(Z,Y)$ of continuous maps between a compact space $Z$ and a metric space $Y$ is the compact-open topology or, equivalently, the topology of uniform convergence. This topology is induced by the metric 
\[\dist (f_1,f_2)=\sup_{z\in Z}d_Y (f_1 z,f_2 z).\]
For a compact metric space $Z$, the group of self-homeomorphisms $\mathrm{Homeo}(Z)$ is a topological group under the topology of uniform convergence. In general, $\mathrm{Homeo}(Z)$ is not locally compact.

\begin{prop} Endow the M\"obius group of $X$ with the topology of uniform convergence. Then we have the following:
\begin{itemize}
\item[(i)] the metric differentiation map $\Mob(X)\to C(X,\R)$, given by $g\mapsto |g'|$, is continuous;
\item[(ii)] The topological group $\Mob(X)$ is locally compact and $\sigma$-compact. It contains the isometry group $\mathrm{Isom}(X)$ as a compact subgroup;
\item[(iii)] for each $p>D$, the affine isometric action of $\Mob(X)$ on $L^p(X\times X,\nu)$, given by $(g,F)\mapsto g. F+c_g$, is continuous. 
\end{itemize}
\end{prop}

\begin{proof}  We start by showing that
\begin{align}\label{alpha bound} \dist (|g'|, 1)\leq C^{-1}\:\dist (g,\mathrm{id}), \textrm{ provided that } \dist (g,\mathrm{id})\leq C
\end{align} 
for some $C>0$ depending on $X$ only. 

There is $\kappa>0$ such that no two open balls of radius $\kappa$ cover $X$; the easy proof, by contradiction, is left to the reader. Now assume that $D(g):=\dist (g,\mathrm{id})\leq \kappa/10$. Fix $x\in X$. By the defining property of $\kappa$, there are $u,v\in X$ such that $d(x,u)\geq \kappa$, $d(u,v)\geq \kappa$, $d(x,v) \geq \kappa$. Recall from the proof of Lemma~\ref{alternate} the following local formula: 
\begin{align*}
 |g'|(x)=\frac{d(gx,gu)}{d(x,u)}\: \frac{d(gx,gv)}{d(x,v)}\: \frac{d(u,v)}{d(gu,gv)}
 \end{align*}
We have $d(x,u)-2D(g)\leq d(gx,gu) \leq d(x,u)+2D(g)$. As $d(x,u)\geq \kappa$, we obtain
\[1-2\kappa^{-1}D(g)\leq \frac{d(gx,gu)}{d(x,u)} \leq 1+2\kappa^{-1}D(g).\]
The same bounds are valid for $x$ and $v$, and for $u$ and $v$, instead of $x$ and $u$. Therefore
\begin{align*}
\frac{(1-2\kappa^{-1}D(g))^2}{1+2\kappa^{-1}D(g)}-1\leq |g'|(x)-1\leq \frac{(1+2\kappa^{-1}D(g))^2}{1-2\kappa^{-1}D(g)}-1.
\end{align*}
The lower bound is greater than $-6\kappa^{-1}D(g)$, whereas the upper bound is at most $8\kappa^{-1}D(g)$. The claim \eqref{alpha bound} is thus proved, with $C:= \kappa/10$.

(i) By \eqref{alpha bound}, the metric differentiation map is continuous at the identity element of $\Mob(X)$. The continuity on $\Mob(X)$ follows, since $\dist (|g'|, |h'|)\leq (\max |h'|)\: \dist (|(g h^{-1})'|, 1)$ for all $g,h$ in $\Mob(X)$ by using the chain rule \eqref{chain rule}.

(ii) It is clear that $\Mob(X)$ is a topological group, and a closed subgroup of $\mathrm{Homeo}(X)$. For each $R\geq 1$ the subset $\{g\in \Mob(X): \max |g'|\leq R\}$ is closed and equicontinuous, hence compact by the Arzel\`a - Ascoli theorem. On the one hand, it follows that $\Mob(X)$ is $\sigma$-compact. On the other hand, \eqref{alpha bound} implies that the closed ball of radius $C$ around the identity is contained in $\{g\in \Mob(X): \max |g'|\leq 2\}$. Since the latter is compact, the former is a compact neighborhood of the identity, and we conclude that $\Mob(X)$ is locally compact. In what concerns the compactness of the isometry group, note that $\mathrm{Isom}(X)=\{g\in \Mob(X):\max |g'|\leq 1\}$.
 
(iii) By definition, an affine isometric action of a topological group on a Banach space is continuous if the linear part of the action is strongly (that is, pointwise) continuous, and the cocycle is (norm) continuous.

Fix $p>D$. First, we show that the linear isometric action $(g,F)\mapsto g. F$ of $\Mob(X)$ on $L^p(X\times X,\nu)$ is strongly continuous. To that end, it suffices to check that for a dense set of functions $F$ in $L^p(X\times X,\nu)$ we have $g.F\to F$ in $L^p(X\times X,\nu)$ whenever $g\to\mathrm{id}$ in $\Mob(X)$. Recall from Lemma~\ref{nature of nu} that $\nu$ is a Radon measure on the locally compact space $X^{2-}:= X\times X - \mathrm{diag}$. Hence $C_c(X^{2-})$, the subspace of compactly-supported continuous functions on $X^{2-}$, is dense in $L^p(X^{2-},\nu)=L^p(X\times X,\nu)$. Let $F\in C_c(X^{2-})$. If $g\to\mathrm{id}$ in $\Mob(X)$ then $g.F\to F$ uniformly. Since $\|g.F-F\|_{L^p(X^{2-},\nu)}\leq \|g.F-F\|_\infty \: \nu(\mathrm{supp}\: F)$ and $\nu(\mathrm{supp}\: F)<\infty$, we conclude that $g.F\to F$ in $L^p(X\times X,\nu)$.

To show that the cocycle $g\mapsto c_g$ is continuous, it suffices to check continuity at the identity element of $\Mob(X)$; the cocycle rule will then imply continuity at every element of $\Mob(X)$. Let $r \in (D,p)$. Then
\[\|c_g\|^p_{L^p(\nu)} \leq \|c_g\|_\infty^{p-r} \|c_g\|^{r}_{L^r(\nu)}.\]
We have
\[ \|c_g\|_\infty=\log \frac{\max |g'|}{\min |g'|}, \qquad \|c_g\|_{L^r(\nu)}\cleq_r\: \frac{\max |g'|}{\min |g'|}, \]
the latter by Lemma~\ref{distance bounds cocycle}. Therefore
\[
\|c_g\|^p_{L^p(\nu)} \cleq_r \Big(\log \frac{\max |g'|}{\min |g'|}\Big)^{p-r}\:\Big(\frac{\max |g'|}{\min |g'|}\Big)^{r},
\]
from which it follows that $\|c_g\|_{L^p(\nu)}\to 0$ as $g\to \mathrm{id}$ in $\Mob(X)$.
\end{proof}

Without further assumptions, the properness of the affine isometric action of $\Mob(X)$ on $L^p(X\times X,\nu)$ seems elusive.

\section{Interlude: Visual metrics on boundaries of hyperbolic groups}
The next goal is to apply the construction of Section~\ref{M-actions} to the action of a non-elementary hyperbolic group on its boundary. In order to do so, we need a metric on the boundary such that the group action is by M\"obius maps, and such that the corresponding Hausdorff measure is Ahlfors regular.

\subsection{Visual metrics induced by the word metric} Let $\G$ be a non-elementary hyperbolic group, and consider the Cayley graph of $\G$ with respect to a finite generating set. Throughout, we choose the identity element as the basepoint.

In the beginning there is the word metric, and the word metric is geodesic. Let $\bd \G$ denote the Gromov boundary of the Cayley graph. Topologically, $\bd \G$ is a canonical compact space without isolated points on which $\G$ acts by homeomorphisms. The metric structure on $\bd\G$ is, however, more subtle. The classical approach runs as follows. First, the Gromov product $(\cdot,\cdot)$ on $\G\times \G$ is extended to $\overline{\G}\times \overline{\G}$, where $\overline{\G}=\G\cup \bd \G$ is the boundary compactification. Such an extension involves a somewhat ad-hoc choice and, a priori, it is neither unique nor continuous. Next, the expression $\exp(-\e(\cdot,\cdot))$ is turned into a compatible metric on the boundary. More precisely, for each sufficiently small parameter $\e>0$ there is a \emph{visual metric} $d_\e\asymp \exp(-\e(\cdot,\cdot))$ which agrees with the canonical topology on $\bd\G$. (See, for instance, Bourdon \cite{B} and Kapovich - Benakli \cite{KB} for more details.) For a visual metric $d_\e$, the geometric mean-value property $(*)$ holds up to multiplicative constants depending on the visual parameter $\e$ and the hyperbolicity constant $\delta$. Clearly, starting with a ``quasi-definition'' of the extended Gromov product snowballs into a quasified metric structure on $\bd\G$. 
 
The Patterson - Sullivan theory for non-elementary hyperbolic groups, due to Coornaert \cite{C}, describes the Hausdorff dimensions and the Hausdorff measures associated to visual metrics on the boundary. Let
\[e(\G)=\limsup_{n\to \infty}\frac{\log\: \#\{g\in \G: l(g)\leq n\}}{n}\]
be the growth exponent of $\G$, where $l$ denotes the word length. Then $0< e(\G)<\infty$, and the following holds:

\begin{prop}[\cite{C}]\label{Hausdorff old} Equip $\bd \G$ with a visual metric $d_\e$. If $D_\e$ denotes the Hausdorff dimension of $\bd \G$, then $D_\e=e(\G)/\e$ and the $D_\e$-dimensional Hausdorff measure is Ahlfors regular.
\end{prop}
 
\subsection{Visual metrics induced by Mineyev's hat metric}\label{this is recalled later} In the $\mathrm{CAT(-1)}$ setting, the metric structure on the boundary is better behaved. Namely, let $X$ be a proper $\mathrm{CAT(-1)}$ space with a fixed basepoint $o\in X$. Then the Gromov product $(\cdot,\cdot)_o$ on $X\times X$ extends continuously to $\overline{X}\times \overline{X}$, and for each $\e\in (0,1]$ the expression $\exp(-\e(\cdot,\cdot)_o)$ is a compatible metric on the boundary $\bd X$ (see Bourdon \cite{B}). If we equip $\bd X$ with a visual metric $\exp(-\e(\cdot,\cdot)_o)$, then the action of $\mathrm{Isom}(X)$ on $\bd X$ is M\"obius. Furthermore, the M\"obius group of $\bd X$ does not depend on the visual parameter $\e$.
 
Mineyev showed in \cite{M} that properties similar to the ones in the $\mathrm{CAT(-1)}$ setting can be achieved on boundaries of hyperbolic groups. The main technical point is the replacement of the word metric on the group by a new metric, herein called the \emph{hat metric}. First introduced by Mineyev and Yu in \cite{Min02}, the hat metric is the key geometric ingredient in their proof of the Baum - Connes conjecture for hyperbolic groups.

As before, let $\G$ be a non-elementary hyperbolic group. Let $X$ be the Cayley graph of $\G$ with respect to a finite generating set, and endow $X$ with the path metric $d$. The following result collects the properties of the hat metric which are relevant for this paper.

\begin{prop}[\cite{M, M-}]\label{mineyev's theorem}
There is a metric $\dhat$ on $X$ having the following properties:
\begin{itemize}
\item[(i)] $\dhat$ is $\G$-invariant, quasi-isometric to $d$, and roughly geodesic;
\item[(ii)] for each $o\in X$, the Gromov product $\la\cdot,\cdot\ra_o$ with respect to $\dhat$ on $X$ extends to a continuous map $\la\cdot,\cdot\ra_o: \overline{X}\times\overline{X}\to [0,\infty]$, where $\la \xi,\omega\ra_o=\infty$ if and only if $\xi=\omega\in \bd X$;
\item[(iii)] $\mathrm{(normalization)}$ for each $\e\in (0,1]$, $\exp(-\e\la\cdot,\cdot\ra_o)$ is a metric on $\bd X$ for all $o\in X$.
 \end{itemize}
\end{prop}

Thus, by passing from the path metric $d$ to the hat metric $\dhat$, we get a Gromov product which behaves just like the one in the $\mathrm{CAT(-1)}$ setting. The metrics on the boundary $\bd X=\bd \G$ are then simple and explicit, and they also have sharp properties. In the direction that concerns us, we have that the action on the boundary is M\"obius, whereas classically the action is quasi-M\"obius only. Recall, the identity element of $\G$ is our chosen basepoint.

\begin{cor}[\cite{M}]\label{look, now it's mobius}
Equip $\bd\G$ with a visual metric $\dhat_\e=\exp(-\e \la \cdot,\cdot\ra)$ defined by the Gromov product with respect to $\dhat$. Then $\G$ acts on $\bd \G$ by M\"obius homeomorphisms.
\end{cor}
Note also that the M\"obius group of $\bd \G$ is independent of the choice of visual parameter $\e$. 

\begin{proof}
We verify that the geometric mean-value property $(*)$ holds for each $g\in \G$. A direct calculation shows that 
\[2\la gx, gw \ra-2\la x,w\ra=\big(\lhat(g^{-1})-2\la g^{-1},x\ra\big)+ \big(\lhat(g^{-1})-2\la g^{-1},w\ra\big)\qquad (x,w\in \G),\] 
where $\hat{l}(g):=\dhat(1,g)$ is the hat length of $g$. Letting $x\to \xi$ and $w\to\omega$, we get
\begin{align}\label{longish formula}
\dhat_\e^{\: 2}(g\xi,g\omega)=\exp\Big(\e\big(2\la g^{-1},\xi\ra-\lhat(g^{-1})\big)\Big)\:\exp\Big(\e\big(2\la g^{-1},\omega\ra-\lhat(g^{-1})\big)\Big)\:\dhat_\e^{\: 2}(\xi,\omega)
\end{align}
for all $\xi,\omega\in \bd\G$.
\end{proof}

Mineyev's hat metric straightens the outside while wrinkling the inside. Indeed, the previous corollary witnesses the fact that exchanging the path metric $d$ for the hat metric $\dhat$ improves the metric structure on the boundary. However, there is a price to be paid within the space $X$: while the path metric $d$ is geodesic, the hat metric $\dhat$ is only roughly geodesic. Recall, to say that $\dhat$ is \emph{roughly geodesic} ($^+$geodesic, in the language of \cite{M}) is to mean the following: there is a spatial constant $C\geq 0$ with the property that, for any points $x,y\in X$, there is a (not necessarily continuous) map $\gamma: [a,b]\to \R$ such that $\gamma(a)=x$, $\gamma(b)=y$, and $\gamma$ is a $C$-rough isometry, i.e., 
\[|s-t|-C\leq \dhat\big(\gamma(s),\gamma(t)\big)\leq |s-t|+C\] 
for all $s,t\in [a,b]$. Hyperbolicity is a quasi-isometry invariant for roughly geodesic spaces, so $\dhat$ is a hyperbolic metric on $X$.

Recent work of Blach\`ere, Ha\"issinsky and Mathieu shows that we still have Ahlfors regularity for the visual metrics coming from the hat metric. More precisely, \cite[Thm.2.3]{BHM} leads to the following:

\begin{prop}\label{Hausdorff new} Equip $\bd \G$ with a visual metric $\dhat_\e=\exp(-\e \la \cdot,\cdot\ra)$ defined by the Gromov product with respect to $\dhat$. If $\hat{D}_\e$ denotes the Hausdorff dimension of $\bd \G$, then $\hat{D}_\e=\hat{e}(\G)/\e$ and the $\hat{D}_\e$-dimensional Hausdorff measure is Ahlfors regular.
\end{prop}
Here 
\[\hat{e}(\G)=\limsup_{n\to \infty}\frac{\log\: \#\{g\in \G: \hat{l}(g)\leq n\}}{n}\]
is the growth exponent of $\G$ with respect to the hat length $\hat{l}$.

To conclude, Corollary~\ref{look, now it's mobius} and Proposition~\ref{Hausdorff new} fulfill the two desiderata announced at the beginning of the section. 

\begin{rem} In \cite{BHM}, the authors are concerned with another interesting metric on a hyperbolic group, namely the Green metric coming from a random walk. The Green metric is a more natural, and much easier to construct, replacement of the word metric than the hat metric. It would be interesting to know whether the Green metric leads to a metric structure on the boundary as nice as the one coming from the hat metric.
\end{rem}

\section{Proper isometric actions of hyperbolic groups on $L^p$-spaces}\label{hyperbolic case}
Again, $\G$ is a non-elementary hyperbolic group. Equip $\bd \G$ with a visual metric $\dhat_\e=\exp(-\e \la \cdot,\cdot\ra)$ defined by the Gromov product with respect to the hat metric $\dhat$. Then, as explained in the previous section, the action of $\G$ on $\bd\G$ falls under the framework developed in Section~\ref{M-actions}. Let us see what we obtain.

\subsection{A $\G$-invariant measure on $\bd\G\times \bd\G$} \label{the paragraph on measure}
While the Hausdorff dimension of $\bd\G$ depends on $\e$, the Hausdorff measure on $\bd\G$ does not, and we simply denote it by $\mu$. The M\"obius-invariant measure $\nu$ on $\bd\G\times \bd\G$ is also independent of the visual parameter, and it takes the form
\begin{align}\label{a group-invariant measure}
\D \nu(\xi,\omega)=\exp\big(2\hat{e}(\G)\:\la\xi,\omega \ra\big)\D\mu (\xi)\D\mu (\omega).
\end{align}
In particular, $\nu$ is $\G$-invariant. Since the diagonal of $\bd\G\times \bd\G$ is $\nu$-negligible, we may also view $\nu$ as a $\G$-invariant Radon measure on $\bd^2\G:=\bd\G\times \bd\G- \mathrm{diag}$.

The measure $\nu$ is a generalization of the Bowen--Margulis measure from the CAT$(-1)$ setting. Namely, consider the special case when $\G$ acts geometrically on a proper CAT$(-1)$ space $X$. If we forgo the visual metrics coming from the hat metric on the Cayley graph of $\G$, and we use instead the visual metrics coming from within $X$, then the M\"obius-invariant measure $\nu$ we obtain is the so-called \emph{Bowen--Margulis measure}. This has the form
\[\D \nu_\mathrm{BM}(\xi,\omega)=\frac{\D\mu_o (\xi)\D\mu_o (\omega)}{d_{\e,o}(\xi,\omega)^{2D_\e}}=\exp\big(2e(\G)\:(\xi,\omega)_o\big)\D\mu_o (\xi)\D\mu_o (\omega),\]
independent of the basepoint $o\in X$. Here $D_\e$ and $\mu_o$ denote the Hausdorff dimension, respectively the Hausdorff measure, with respect to the visual metric $d_{\e,o}=\exp(-\e(\cdot,\cdot)_o)$.

The existence of a $\G$-invariant measure on $\bd\G\times \bd\G$ analogous to the Bowen--Margulis measure was first established by Furman in \cite[Prop.1]{Fu} using a cohomological argument. Roughly speaking, Furman constructs the measure in the loose measurable sense whereas our $\nu$ is constructed in the sharp metric category. This grants $\nu$ some advantages: it is more explicit, it is much closer to the Bowen--Margulis measure, and it has the sharp properties needed for the construction of a proper isometric action on $L^p(\bd\G\times\bd\G)$.

\subsection{A cocycle for the $\G$-action on $\bd\G\times \bd\G$}\label{the paragraph on cocycle} By \eqref{longish formula}, the metric derivative of $g\in \G$ is
\begin{align}\label{here is the derivative}
|g'|^{(\e)}(\xi)=\exp\Big(\e\big(2\la g^{-1},\xi\ra-\lhat(g^{-1})\big)\Big),
\end{align}
and the cocycle $c$ becomes, up to a factor of $2\e$, 
\begin{align}\label{what a pretty cocycle}
g\mapsto c_g(\xi,\omega)=\la g,\xi\ra-\la g,\omega\ra.
\end{align}
We may interpret this attractive cocycle as follows. For $g\in \G$ and $\xi\in\bd\G$, let $\hat{\beta}(g,\xi):=2\la g,\xi\ra-\lhat(g)$. If we fix a boundary point $\xi$ and we view $\hat{\beta}$ as a function on the group, then $\hat{\beta}(\cdot,\xi)$ is the \emph{Busemann function} corresponding to $\xi$, and the difference map $(g,h)\mapsto \hat{\beta}(g,\xi)- \hat{\beta}(h,\xi)$ is the \emph{Busemann cocycle} with respect to $\xi$. If we fix a group element $g$ and we view $\hat{\beta}$ as a function on the boundary, then the difference map $(\xi,\omega)\mapsto \hat{\beta}(g,\xi)- \hat{\beta}(g,\omega)$ is, up to a factor of $2$, our cocycle $c_g$. We thus think of $c$ as the \emph{other Busemann cocycle}.

Let $p>\hat{D}_\e$; we recall that $\hat{D}_\e$ denotes the Hausdorff dimension of $(\bd \G, \dhat_\e)$. By Proposition~\ref{cocycle}, $c$ is a cocycle for the linear isometric action of $\G$ on $L^p(\bd\G\times\bd\G)$. The last ingredient is the following:

\begin{prop} The cocycle $c$ is proper. In fact, the growth of $g\mapsto \|c_g\|^p_{L^p(\nu)}$ is at least linear with respect to the (hat or word) length.
\end{prop}

\begin{proof} 
Let $\hat{\delta}$ denote the hyperbolicity constant of $\dhat$. Put $K:=e^{\hat{e}(\G)}$. Then 
\begin{align*}
\|c_g\|^p_{L^p(\nu)}&=\iint |\la g,\xi\ra- \la g,\omega\ra|^p\: K^{\:2\la \xi,\omega\ra} \D \mu (\xi)\D \mu (\omega)\\
& \cgeq \iint |\la g,\xi\ra- \la g,\omega\ra|^p\: K^{\: 2\min\{\la g,\xi\ra, \la g,\omega\ra\}} \D \mu (\xi)\D \mu (\omega)
\end{align*}
by using the hyperbolic inequality $\la \xi,\omega\ra\geq \min\big\{\la g,\xi\ra, \la g,\omega\ra\big\}-\hat\delta$.

Let $M(g):=\max \{\la g,\xi \ra: \xi\in \bd\G\}$. First, we claim that $\mu\big(\{\xi\in \bd \G\: : \: \la g,\xi\ra\geq R\}\big)\asymp_{\e} K^{-R}$ for $0\leq R\leq M(g)-\hat\delta$. To prove the claim, pick $\omega\in \bd\G$ such that $\la g,\omega\ra=M(g) \geq R+\hat\delta$. Then
\begin{align}\label{shadow}
B(\omega, e^{-\e(R+\hat\delta)})\subseteq \{\xi\in \bd \G: \la g,\xi\ra\geq R\}\subseteq B(\omega, e^{-\e(R-\hat\delta)})
\end{align}
where the balls are closed, and taken with respect to the $\dhat_\e$ metric. Indeed, let $\xi\in B(\omega, e^{-\e(R+\hat\delta)})$; then $\la \xi,\omega\ra\geq R+\hat\delta$, hence $\la g,\xi\ra\geq \min\big\{\la g,\omega\ra, \la \xi,\omega\ra\big\}-\hat\delta\geq R$. This justifies the first inclusion in \eqref{shadow}. For the second one, let $\xi\in \bd\G$ with $\la g,\xi\ra \geq R$; then $\la \xi,\omega\ra\geq \min\big\{\la g,\xi\ra, \la g,\omega\ra\big\}-\hat\delta\geq R-\hat\delta$, hence $\dhat_\e(\xi,\omega)\leq e^{-\e(R-\hat\delta)} $. The claimed measure-theoretic estimate follows from \eqref{shadow} together with the Ahlfors regularity of $\mu$.

From the above claim, it follows that there exists a positive constant $\rho=\rho(\e)$ such that the sets
\[A_k:=\{\xi\in \bd \G\: : \: (k+1) \rho> \la g,\xi\ra\geq k \rho\}\]
satisfy $\mu(A_k)\asymp_{\e} K^{-k\rho}$ for $0\leq k\leq N(g)$, where $N(g):=\big\lfloor \big(M(g)-\hat\delta\big)/\rho\big\rfloor-1$. Splitting over the annular sets $A_k$, we have:
\begin{align*}
\|c_g\|^p_{L^p(\nu)} \cgeq \sum_{0\leq j<k\leq N(g)} \iint_{\begin{smallmatrix} \xi\in A_{k}\\ \omega\in A_{j} \end{smallmatrix}} |\la g,\xi\ra- \la g,\omega\ra|^p\: K^{\: 2\min\{\la g,\xi\ra, \la g,\omega\ra\}} \D \mu (\xi)\D \mu (\omega)\end{align*}
For $\xi\in A_k$ and $\omega\in A_{j}$ with $k>j$, we have $|\la g,\xi\ra- \la g,\omega\ra|=\la g,\xi\ra- \la g,\omega\ra\geq  (k-j-1)\rho$ and $\min\{\la g,\xi\ra, \la g,\omega\ra\}=\la g,\omega\ra\geq j\rho$. Hence
\begin{align*}
\|c_g\|^p_{L^p(\nu)}&\cgeq_{\e} \sum_{0\leq j<k\leq N(g)} (k-j-1)^p \:K^{\: 2j\rho}\:K^{-k\rho}\:K^{-j\rho}\\
&=K^{-\rho}\sum_{0\leq j\leq k\leq N(g)-1} (k-j)^p\: K^{-(k-j)\rho} \geq K^{-2\rho}(N(g)-1),
\end{align*}
where the last inequality follows from a simple recurrence, as in the proof of Lemma~\ref{properness for free}. Summarizing, we have that $\|c_g\|^p_{L^p(\nu)}\geq a\: M(g)-b$ for some positive constants $a=a(\e), b=b(\e)$.

Finally, we claim that $M(g)+M(g^{-1})\geq \lhat(g)$. As $\|c_{g^{-1}}\|_{L^p(\nu)}=\|-g^{-1}.\:c_g\|_{L^p(\nu)}=\|c_g\|_{L^p(\nu)}$, the desired lower bound
\[\|c_g\|^p_{L^p(\nu)}\geq \frac{a}{2}\:\lhat(g)-b\]
immediately follows. As for the claim, it is a consequence of the fact that $\la g,\xi\ra+\la g^{-1},g^{-1}\xi\ra=\lhat(g)$ for all $\xi\in \bd\G$. This can be seen by plugging in formula \eqref{here is the derivative} in the identity $|(g^{-1})'|=1/g.|g'|$. More directly, one can check that $\la g,x\ra+\la g^{-1},g^{-1}x\ra=\lhat(g)$ for all $x\in \G$. \end{proof}

\subsection{Hyperbolic dimension}\label{the paragraph on dimension} The best integrability exponent $p$ in sight is obtained by taking the visual parameter $\e$ to be $1$. Then we obtain a proper action of $\G$ on $L^p(\bd\G\times\bd\G)$ for every $p>\hat{e}(\G)$; recall, $\hat{e}(\G)$ is the growth exponent of $\G$ with respect to the hat metric $\dhat$. However, we can do slightly better. Consider the infimum of all Hausdorff dimensions of the boundary with respect to visual metrics $\exp(-\e \la \cdot,\cdot\ra)$ which come from some hat metric, i.e., a metric satisfying properties i), ii), and iii) of Theorem~\ref{mineyev's theorem}. This infimum is Mineyev's \emph{hyperbolic dimension} (cf. \cite[Section 10]{M}). The notion of hyperbolic dimension is inspired by Pansu's conformal dimension, and one can see that the hyperbolic dimension is at least as large as the the Ahlfors regular, conformal dimension mentioned in the introduction.

Let $\hbar(\G)$ denote the hyperbolic dimension of a non-elementary hyperbolic group $\G$. Then $\hbar(\G)$ is a non-negative real number equal to $0$ if $\G$ is virtually free, and greater or equal to $1$ otherwise \cite[Thm.22 \& Cor.23]{M}. 

\begin{ex}
Consider the rank-$1$ symmetric space $\mathbb{H}^n_K$, where $K=\R, \C, \mathbb{H}$ and $n\geq 2$, or $K=\mathbb{O}$ and $n=2$, normalized so that the maximal sectional curvature is $-1$. If $\G$ is a cocompact lattice in $\mathrm{Isom}(\mathbb{H}^n_K)$ then, by a result of Pansu \cite{P}, we get that $\hbar(\G)=nk+k-2$ where $k=\dim_\R K$. 
\end{ex}

With the notion of hyperbolic dimension at hand, our main result can be neatly stated as follows:

\begin{thm}\label{neat}
Let $\G$ be a non-elementary hyperbolic group. Then $\G$ admits a proper affine isometric action on $L^p(\bd\G\times\bd\G)$ for all $p>\hbar(\G)$. 
\end{thm}
In particular, Theorem~\ref{neat} together with the Fisher - Margulis result that a group with property (T) is $L^{2+\e}$-rigid (see \cite[Thm.1.3]{BFGM}) imply the following consequence: if $\G$ is an infinite hyperbolic group with property (T), then $\hbar(\G)>2$.

\section{An $\ell^p$-cohomological interpretation}\label{lp-cohomology stuff}
In this section we prove Theorem~\ref{lp consequence} from the introduction.

\subsection{Besov spaces} Let us return to the setup of Section~\ref{M-actions}: $(X,d)$ is a compact metric space of Hausdorff dimension $D\in (0,\infty)$, whose $D$-dimensional Hausdorff measure $\mu_D$ is Ahlfors regular.

Following Bourdon and Pajot \cite{BP}, we consider the \emph{Besov space} $B_p(X)$. The Besov seminorm of a measurable function $f:X\to\R$ is given by
\begin{align*}
\|f\|_{B_p}=\Big(\iint \big|f(x)-f(y)\big|^p\: d^{\: -2D}(x,y)\D\mu_D(x)\D\mu_D(y)\Big)^\frac{1}{p}.
\end{align*}
Note that $\|f\|_{B_p}=0$ if and only if $f$ is a.e. constant. The Besov space $B_p(X)$ is the space of functions having finite Besov seminorm, modulo a.e. constant functions. Equipped with the induced Besov norm, $B_p(X)$ is a Banach space. There is an obvious isometric embedding
\begin{align*}
B_p(X) \; \longrightarrow \; L^p(X\times X, \nu), \qquad [f]\; \longmapsto\;  \big( (x,y)\mapsto f(x)-f(y) \big)
\end{align*}
where, we recall, $\D \nu(x,y)=d^{\: -2D}(x,y) \D\mu_D (x)\D\mu_D (y)$. For $p>D$, we also have a continuous embedding 
\[\mathrm{Lip}(X)\; \hookrightarrow B_p(X)\] 
where $\mathrm{Lip}(X)$ is the space of Lipschitz functions modulo constants, equipped with the induced Lipschitz norm. Indeed, if $f$ is Lipschitz on $X$ then $\|[f]\|_{B_p}\leq \|[f]\|_\mathrm{Lip}\|d\|_{L^p(\nu)}$, and Lemma~\ref{integrability of distance} says that $d\in L^p(X\times X,\nu)$ whenever $p>D$.

Now we see that the affine isometric action of the M\"obius group $\Mob(X)$ on the $L^p$-space $L^p(X\times X)$ is, in essence, an affine isometric action on the Besov space $B_p(X)$. More precisely, for each $p>D$, the map $g\mapsto \log |(g^{-1})'|$ is a (Lipschitz) cocycle for the linear isometric action of $\Mob(X)$ on $B_p(X)$.

Moving to the context of hyperbolic groups, recall that the boundary $\bd \G$ of a non-elementary hyperbolic group $\G$ is metrized by some visual metric $\dhat_\e=\exp(-\e \la \cdot,\cdot\ra)$. However, \eqref{a group-invariant measure} shows that the $\G$-invariant Besov norm on the boundary $\bd \G$ is given by 
\begin{align*}
\|f\|_{B_p}=\Big(\iint \big|f(\xi)-f(\omega)\big|^p\: \exp\big(2\hat{e}(\G)\:\la\xi,\omega \ra\big)\D\mu (\xi)\D\mu (\omega)\Big)^\frac{1}{p},
\end{align*}
independent of the choice of visual parameter $\e$. By \eqref{here is the derivative}, the cocycle $g\mapsto \log |(g^{-1})'|$ becomes $g\mapsto \hat{\beta}(g,\cdot)=2\la g,\cdot\ra-\lhat(g)$. This too could be called the Busemann cocycle (compare the discussion in Section~\ref{the paragraph on cocycle}), and it is a Lipschitz cocycle for any choice of visual metric $\dhat_\e=\exp(-\e \la \cdot,\cdot\ra)$. When we work modulo constant functions, as we do in the Lipschitz space $\mathrm{Lip}$ and in the Besov spaces $B_p$, this simple Busemann cocycle simplifies further to $g\mapsto [\la g,\cdot\ra]$. We thus have:

\begin{prop}\label{besov reinterpretation}
Let $\G$ be a non-elementary hyperbolic group. Then the linear isometric action of $\G$ on $B_p(\bd\G)$ admits $g\mapsto [\la g,\cdot\ra]$ as a proper cocycle for all $p>\hat{e}(\G)$.
\end{prop}

\subsection{$\ell^p$-cohomology in degree one} Let $\G$ be a finitely generated group, and consider the Cayley graph of $\G$ with respect to a finite generating set. We let
\[E_p(\G)=\bigg\{\phi:\G\to \R : \|\phi\|_{E_p}=\Big(\sum_{x \bullet\!\textrm{---}\!\bullet y}|\phi(x)-\phi(y)|^p\Big)^\frac{1}{p}<\infty\bigg\}\]
be the linear space of functions on $\G$ with $p$-summable edge differential. The \emph{first $\ell^p$-cohomology group} $H^1_{(p)}(\G)$ is defined as the quotient of $E_p(\G)$ by $\ell^p\G+\R$.

The $E_p$-seminorms with respect to two Cayley graphs of $\G$ are comparable, meaning that the linear space $E_p(\G)$ and the isomorphism type of the seminormed space $(E_p(\G), \|\cdot\|_{E_p})$ are both independent of the choice of Cayley graph for $\G$. Furthermore, the $E_p$-seminorm is $\G$-invariant. All these features are inherited by $H^1_{(p)}(\G)$, equipped with the induced norm. If $\G$ is non-amenable, then $H^1_{(p)}(\G)$ is a Banach space.

Now let us turn again to non-elementary hyperbolic groups. Applying a beautiful result of Bourdon and Pajot \cite[Thm.0.1, Thm.3.4]{BP} to our situation, we obtain the following fact:

\begin{prop}\label{lp versus besov}
 Let $\G$ be a non-elementary hyperbolic group. Then the following hold:
\begin{itemize}
\item[(i)] for each $\phi\in E_p(\G)$, the boundary extension $\phi_\infty(\xi)=\lim_{x\to \xi} \phi(x)$ is a.e. defined on $\bd \G$;
\item[(ii)] the boundary extension induces a $\G$-equivariant Banach space isomorphism $H^1_{(p)}(\G)\to B_p(\bd\G)$, given by $[\phi]\mapsto [\phi_\infty]$.
\end{itemize}
\end{prop}

In light of the equivariant isomorphism between the first $\ell^p$-cohomology group and the boundary Besov space, Proposition~\ref{besov reinterpretation} can be stated as follows:

\begin{thm}\label{final theorem}
Let $\G$ be a non-elementary hyperbolic group. Then the linear isometric action of $\G$ on $H^1_{(p)}(\G)$ admits $g\mapsto [\la g,\cdot\ra]$ as a proper cocycle for all $p>\hat{e}(\G)$.
\end{thm}

Note that the cocycle in Theorem~\ref{final theorem} is given by equivalence classes of functions on $\G$, whereas the similar-looking cocycle in Proposition~\ref{besov reinterpretation} is given by equivalence classes of functions on $\bd\G$. There should be a direct proof for Theorem~\ref{final theorem}, in fact one involving the Gromov product $(\cdot, \cdot)$ with respect to the word metric on $\G$ rather than the hat Gromov product $\la\cdot, \cdot\ra$.

\section*{Acknowledgements}
Part of this work was carried out at the University of Victoria (Canada), where I was supported by a Postdoctoral Fellowship from the Pacific Institute for the Mathematical Sciences (PIMS). I thank Yves de Cornulier and Guoliang Yu for comments on a preliminary version of this paper. I am particularly grateful to Marc Bourdon for pointing out that my construction has an $\ell^p$-cohomological interpretation. Finally, I thank the referee for useful remarks.

\end{document}